\documentclass[oneside,reqno,english]{amsart}
\usepackage[T1]{fontenc}
\usepackage[latin9]{inputenc}
\usepackage{babel}
\usepackage{prettyref}
\usepackage{mathrsfs}
\usepackage{enumitem}
\usepackage{amstext}
\usepackage{amsthm}
\usepackage{amssymb}
\usepackage[all]{xy}
\usepackage[pdfusetitle,
 bookmarks=true,bookmarksnumbered=false,bookmarksopen=false,
 breaklinks=false,pdfborder={0 0 1},backref=false,colorlinks=false]
 {hyperref}
\hypersetup{
 colorlinks=true,citecolor=blue,linkcolor=blue,linktocpage=true}

\makeatletter
\numberwithin{equation}{section}
\numberwithin{figure}{section}
\theoremstyle{plain}
\newtheorem{thm}{\protect\theoremname}[section]
\theoremstyle{plain}
\newtheorem{cor}[thm]{\protect\corollaryname}
\theoremstyle{remark}
\newtheorem{rem}[thm]{\protect\remarkname}
\theoremstyle{definition}
\newtheorem{example}[thm]{\protect\examplename}
\theoremstyle{definition}
\newtheorem{defn}[thm]{\protect\definitionname}
\theoremstyle{plain}
\newtheorem{lem}[thm]{\protect\lemmaname}

\newrefformat{cor}{Corollary~\ref{#1}}
\newrefformat{subsec}{Section~\ref{#1}}
\newrefformat{lem}{Lemma~\ref{#1}}
\newrefformat{thm}{Theorem~\ref{#1}}
\newrefformat{sec}{Section~\ref{#1}}
\newrefformat{chap}{Chapter~\ref{#1}}
\newrefformat{prop}{Proposition~\ref{#1}}
\newrefformat{exa}{Example~\ref{#1}}
\newrefformat{tab}{Table~\ref{#1}}
\newrefformat{rem}{Remark~\ref{#1}}
\newrefformat{def}{Definition~\ref{#1}}
\newrefformat{fig}{Figure~\ref{#1}}
\newrefformat{claim}{Claim~\ref{#1}}

\AtBeginDocument{
  
}

\makeatother

\providecommand{\corollaryname}{Corollary}
\providecommand{\definitionname}{Definition}
\providecommand{\examplename}{Example}
\providecommand{\lemmaname}{Lemma}
\providecommand{\remarkname}{Remark}
\providecommand{\theoremname}{Theorem}

\begin{document}
\title[]{Hilbert space valued Gaussian processes, their kernels, factorizations,
and covariance structure}
\author{Palle E.T. Jorgensen}
\address{(Palle E.T. Jorgensen) Department of Mathematics, The University of
Iowa, Iowa City, IA 52242-1419, U.S.A.}
\email{palle-jorgensen@uiowa.edu}
\author{James Tian}
\address{(James F. Tian) Mathematical Reviews, 416 4th Street Ann Arbor, MI
48103-4816, U.S.A.}
\email{jft@ams.org}
\begin{abstract}
Motivated by applications, we introduce a general and new framework
for operator valued positive definite kernels. We further give applications
both to operator theory and to stochastic processes. The first one
yields several dilation constructions in operator theory, and the
second to general classes of stochastic processes. For the latter,
we apply our operator valued kernel-results in order to build new
Hilbert space-valued Gaussian processes, and to analyze their structures
of covariance configurations.
\end{abstract}

\maketitle

\section{Introduction}

Our present paper deals with two closely related issues, (i) a number
of questions from operator theory, factorizations, and dilations;
and from (ii) non-commutative stochastic analysis. The first one addresses
questions around classes of Hilbert space dilations; and the second
(ii) deals with a general framework for Hilbert space valued Gaussian
processes. While, on the face of it, it is not obvious that the two
questions are closely related, but as we show, (i) and (ii) together
suggest the need for an extension of both the traditional framework
of positive definite (scalar valued) kernels and their associated
reproducing kernel Hilbert spaces (RKHSs); and of the framework for
non-commutative stochastic analysis. Hence our present starting point
will be a fixed separable Hilbert space $H$, and we then study the
notion of positive definite kernels $K$ which take values in the
algebra of all bounded operators on $H$, here denoted $B(H)$. 

One of the advantages of our new perspective is reflected in the following
important fact: a variety of Hilbert completions in our approach will
now take an explicit form, as opposed the guise of abstract equivalence
classes. In particular, every Hilbert space arising via a GNS construction
will now instead be a concrete RKHS. Hence, explicit formulas, and
no more abstract and inaccessible equivalence classes.

For the theory or reproducing kernels, the literature is extensive,
both pure and recent applications, we refer to the following, including
current, \cite{MR80878,MR31663,MR4302453,MR3526117,MR1305949,MR4690276,MR3700848}.
Regarding more recent applications of RKHSs, we refer the reader to
e.g., \cite{MR4575370}.

\textbf{Notations.} Throughout the paper, we continue with the physics
convention: inner products linear in the second variable. 

For a p.d. kernel $K:S\times S\rightarrow\mathbb{C}$, we denote by
$H_{K}$ the corresponding reproducing kernel Hilbert space. It is
the Hilbert completion of 
\[
span\left\{ K_{y}\left(\cdot\right):=K\left(\cdot,y\right)\mid y\in S\right\} 
\]
with respect to the inner product 
\[
\left\langle \sum_{i}c_{i}K\left(\cdot,x_{i}\right),\sum_{i}d_{j}K\left(\cdot,x_{j}\right)\right\rangle _{H_{K}}:=\sum_{i}\sum_{j}\overline{c}_{i}d_{j}K\left(x_{i},x_{j}\right).
\]
The reproducing property is as follows: 
\begin{equation}
\varphi\left(x\right)=\left\langle K\left(\cdot,x\right),\varphi\right\rangle _{H_{K}},\forall x\in S,\:\forall\varphi\in H_{K}.\label{eq:A3}
\end{equation}
For any orthonormal basis (ONB) $\left(\varphi_{i}\right)$ in $H_{K}$,
one has 
\begin{equation}
K\left(x,y\right)=\sum_{i}\varphi_{i}\left(x\right)\overline{\varphi_{i}\left(y\right)},\quad\forall x,y\in S.\label{eq:A4}
\end{equation}

A $B\left(H\right)$-valued kernel $K:S\times S\rightarrow B\left(H\right)$
is positive definite (p.d.) if 
\begin{equation}
\sum_{i=1}^{n}\left\langle a_{i},K\left(s_{i},s_{j}\right)a_{j}\right\rangle _{H}\geq0\label{eq:A5}
\end{equation}
for all $\left(a_{i}\right)_{1}^{n}$ in $H$, and all $n\in\mathbb{N}$.

\section{\protect\label{sec:2}Operator valued p.d. kernels}

While operator theory, and analysis more generally, is rich with theories
and applications where positive definite kernels (p.d.) of one kind
or the other (e.g., \cite{MR1200633,MR3526117,MR4280112}) play a
crucial role , the purpose of the present section is to point out
a useful duality framework for the case of operator valued p.d. kernels,
here referred to as $B(H)$-valued kernels, on the one hand, and a
scalar valued p.d. kernel on the other. As noted in the cited references,
a recurrent theme for this framework is the construction of dilation
Hilbert spaces, and associated systems of isometries;---the idea
being that efficient choices of dilations serve to encapsule \textquotedblleft better\textquotedblright{}
solutions, where \textquotedblleft better\textquotedblright{} typically
refers to spectral theoretic properties. The duality discussed here
(see especially \prettyref{thm:B1} and \prettyref{cor:B2}) implies
an efficient way of identifying such dilations. Starting with a $B(H)$-valued
p.d. kernel $K$ for a general set $S$ (or rather $S\times S$),
we identify a canonical scalar valued p.d. kernel $\tilde{K}$ but
now for the set $S\times H$. The gist of this duality for $K$ and
$\tilde{K}$ (see \eqref{eq:B1}) is that both the dilations and the
corresponding system of intertwining operators (see \eqref{eq:B2})
take a natural form, and they allow for explicit formulas, as opposed
to just abstract existence. In \prettyref{thm:B1} and Corollaries
\ref{cor:B2}, \ref{cor:B3}, and \ref{cor:B4}, we flesh out the
details regarding this point, see \eqref{eq:B9} and \eqref{eq:B10}.
And in subsection \ref{subsec:b1} we further offer examples. \prettyref{sec:3}
below in turn deals with a related framework but now in the context
of Schwartz distributions. We stress that our present general framework
is motivated by questions for Gaussian processes in stochastic analysis,
and in the corresponding quantum counterparts; see e.g., \cite{MR3402823,MR489494,MR4509531}
and \cite{MR4280112}. As noted in the literature, there is a natural
correspondence between p.d. kernels on the one hand and Gaussian processes
on the other. In particular, every p.d. kernel is the covariance kernel
of a Gaussian process.

Below we introduce the general framework for operator valued positive
definite kernels. Our analysis is applied to several dilation constructions,
both in operator theory, and in stochastic processes. The latter will
be taken up in \prettyref{sec:4} again where we apply the results
here to build Hilbert space valued Gaussian processes, and examine
their properties. 

Let $S$ be a set, and let $K:S\times S\rightarrow B\left(H\right)$
be positive definite. Set $X=S\times H$, and define $\tilde{K}:X\times X\rightarrow\mathbb{C}$
by 
\begin{equation}
\tilde{K}\left(\left(s,a\right),\left(t,b\right)\right)=\left\langle a,K\left(s,t\right)b\right\rangle _{H},\label{eq:B1}
\end{equation}
for all $a,b\in H$, and all $s,t\in S$. Then $\tilde{K}$ is a scalar-valued
p.d. kernel on $X\times X$. Let $H_{\tilde{K}}$ be the corresponding
RKHS. 

It will follow from our subsequent results that this $H_{\tilde{K}}$
realization has properties that makes it universal. Moreover, one
notes that in different but related p.d. settings, analogous $H_{\tilde{K}}$-constructions
serve as a generators of dilation Hilbert spaces, and representations.
It has counterparts for such related constructions as GNS (Gelfand-Naimark-Segal),
and Stinespring (complete positivity), see e.g., \cite{MR4114386,MR4581177,MR4482713}.

The authors have chosen to include a proof sketch for \prettyref{thm:B1}
as well as the corresponding displayed equations. Three reasons: (i)
this will allow us an opportunity to define and introduce the basic
notions, and terminology, which will be used throughout the paper;
(ii) for the benefit of non-specialists, it will make our paper more
self-contained. And, (iii) it will allow us to link the (more familiar)
classical theory for kernels with the new directions which form the
framework for the results presented later, Gaussian processes, and
the quantum framework.
\begin{thm}
\label{thm:B1}The family of operators $\left\{ V_{s}\right\} _{s\in S}$,
where $V_{s}:H\rightarrow H_{\tilde{K}}$, defined as 
\begin{equation}
V_{s}a=\tilde{K}\left(\cdot,\left(s,a\right)\right):X\rightarrow\mathbb{C},\quad a\in H,\label{eq:B2}
\end{equation}
satisfies the following:
\begin{enumerate}
\item For all $s\in S$, and all $a\in H$, 
\[
\left\Vert V_{s}a\right\Vert _{H_{\tilde{K}}}^{2}=\left\langle a,K\left(s,s\right)a\right\rangle _{H}.
\]
\item The adjoint $V_{s}^{*}:H_{\tilde{K}}\rightarrow H$ is determined
by
\[
V_{s}^{*}\tilde{K}\left(\cdot,\left(t,b\right)\right)=K\left(s,t\right)b.
\]
\item $V_{s}V_{s}^{*}:H_{\tilde{K}}\rightarrow H_{\tilde{K}}$ is given
by 
\[
V_{s}V_{s}^{*}\tilde{K}\left(\cdot,\left(t,b\right)\right)=\tilde{K}\left(\cdot,\left(s,K\left(s,t\right)b\right)\right).
\]
\item For all $s_{1},s_{2},\cdots,s_{n}$ in $S$, 
\begin{eqnarray*}
 &  & \left(V_{s_{1}}V_{s_{1}}^{*}\right)\cdots\left(V_{s_{n}}V_{s_{n}}^{*}\right)\tilde{K}\left(\cdot,\left(t,b\right)\right)\\
 & = & \tilde{K}\left(\cdot,\left(s_{1},K\left(s_{1},s_{2}\right)\cdots K\left(s_{n-1},s_{n}\right)K\left(s_{n},t\right)b\right)\right).
\end{eqnarray*}
\item For all $s,s'\in S$, and $b\in H$, 
\[
V_{s'}V_{s}^{*}\tilde{K}\left(\cdot,\left(t,b\right)\right)=\tilde{K}\left(\cdot,\left(s',K\left(s,t\right)b\right)\right).
\]
\item $V_{s}^{*}V_{t}b=K\left(s,t\right)b$, and 
\begin{align*}
\left(V_{s_{1}}^{*}V_{t_{1}}\right)\cdots\left(V_{s_{n}}^{*}V_{t_{n}}\right)b & =K\left(s_{1},t_{1}\right)\cdots K\left(s_{n-1},t_{n-1}\right)K\left(s_{n},t_{n}\right)b.
\end{align*}
\end{enumerate}
If, in addition, 
\[
K\left(s,s\right)=I,\quad\forall s\in S,
\]
then the operators $V_{s}:H\rightarrow H_{\tilde{K}}$ are isometric,
and $V_{s}V_{s}^{*}$ are (selfadjoint) projections in $H_{\tilde{K}}$. 
\end{thm}

\begin{proof}
(1) By definition, 
\[
\left\Vert V_{s}a\right\Vert _{H_{\tilde{K}}}^{2}=\left\langle \tilde{K}\left(\cdot,\left(s,a\right)\right),\tilde{K}\left(\cdot,\left(s,a\right)\right)\right\rangle _{H_{\tilde{K}}}=\left\langle a,K\left(s,t\right)b\right\rangle _{H}.
\]

(2) For all $a,b\in H$, and all $s,t\in S$, since 
\[
\left\langle V_{s}a,\tilde{K}\left(\cdot,\left(t,b\right)\right)\right\rangle _{H_{\tilde{K}}}=\left\langle \tilde{K}\left(\cdot,\left(s,a\right)\right),\tilde{K}\left(\cdot,\left(t,b\right)\right)\right\rangle _{H_{\tilde{K}}}=\left\langle a,K\left(s,t\right)b\right\rangle _{H}
\]
it follows that $V_{s}^{*}\tilde{K}\left(\cdot,\left(t,b\right)\right)=K\left(s,t\right)b$,
as stated.

(3) A direct calculation shows that 
\[
V_{s}V_{s}^{*}\tilde{K}\left(\cdot,\left(t,b\right)\right)=V_{s}K\left(s,t\right)b=\tilde{K}\left(\cdot,\left(s,K\left(s,t\right)b\right)\right).
\]

(4) From the above, 
\begin{align*}
\left(V_{s_{n}}V_{s_{n}}^{*}\right)\tilde{K}\left(\cdot,\left(t,b\right)\right) & =V_{s_{n}}K\left(s_{n},t\right)b=\tilde{K}\left(\cdot,\left(s_{n},K\left(s_{n},t\right)b\right)\right)\\
\left(V_{s_{n-1}}V_{s_{n-1}}^{*}\right)\left(V_{s_{n}}V_{s_{n}}^{*}\right)\tilde{K}\left(\cdot,\left(t,b\right)\right) & =V_{s_{n-1}}K\left(s_{n-1},s_{n}\right)K\left(s_{n},t\right)b\\
 & =\tilde{K}\left(\cdot,\left(s_{n-1},K\left(s_{n-1},s_{n}\right)K\left(s_{n},t\right)b\right)\right)
\end{align*}
and the assertion follows by induction. 

(5) This is immediate: 
\[
V_{s'}V_{s}^{*}\tilde{K}\left(\cdot,\left(t,b\right)\right)=V_{s'}\left(K\left(s,t\right)b\right)=\tilde{K}\left(\cdot,\left(s',K\left(s,t\right)b\right)\right)
\]

(6) Use part (5) and induction. 
\end{proof}
Notably, the above theorem leads to a characterization/factorization
for operator-valued p.d. kernels: 
\begin{cor}
\label{cor:B2}Let $K:S\times S\rightarrow B\left(H\right)$. Then
the kernel $K$ is p.d. if and only if there exists a family of operators
$\left\{ V_{s}\right\} _{s\in S}$ from $H$ to another Hilbert space
$L$, such that $K$ factors as 
\begin{equation}
K\left(s,t\right)=V_{s}^{*}V_{t},\quad s,t\in S.\label{eq:B9}
\end{equation}
If $K\left(s,s\right)=I$ for all $s\in S$, then $V_{s}:H\rightarrow L$
are isometric. 

Furthermore, if $L$ is minimal in the sense that
\[
L=\overline{span}\left\{ V_{s}a:s\in S,a\in H\right\} ,
\]
then $L\simeq H_{\tilde{K}}.$
\end{cor}

\begin{proof}
The ``only if'' part is contained in \prettyref{thm:B1} (6). 

Conversely, assume $K\left(s,t\right)=V_{s}^{*}V_{t}$ as stated.
For all $\left(a_{i}\right)_{i=1}^{n}$ in $H$, and $\left(s_{i}\right)_{i=1}^{n}$
in $S$, 
\begin{align*}
\sum_{i,j}\left\langle a_{i},K\left(s_{i},s_{j}\right)a_{j}\right\rangle _{H} & =\sum_{i,j}\left\langle a_{i},V_{s_{i}}^{*}V_{s_{j}}a_{j}\right\rangle _{H}\\
 & =\sum_{i,j}\left\langle V_{s_{i}}a_{i},V_{s_{j}}a_{j}\right\rangle _{L}=\left\Vert \sum_{i}V_{s_{i}}a_{i}\right\Vert _{L}^{2}\geq0
\end{align*}
which shows that $K$ is positive definite.

Now, under the assumption $L=\overline{span}\left\{ V_{s}a:s\in S,a\in H\right\} $,
the map 
\[
\tilde{K}\left(\cdot,\left(s,a\right)\right)\longmapsto V_{s}a
\]
extends, by linearity and density, to a unique isometric isomorphism
from $H_{\tilde{K}}$ onto $L$, and so any such space $L$ is isomorphic
to the RKHS $H_{\tilde{K}}$.
\end{proof}
\begin{rem}
\label{rem:B3}The maps $V_{s}$ in \prettyref{thm:B1} and \prettyref{cor:B2},
as well as in other contexts, refer to vector-valued evaluation maps.
In the particular scalar case where $H=\mathbb{C}$, it becomes evident
that the evaluation functional appears (c.f. \eqref{eq:B9}):
\[
K\left(x,y\right)=ev_{x}\circ ev_{y}^{*}.
\]
In fact, this scalar version of \prettyref{thm:B1} is well-known
and is a standard method for recovering reproducing kernels. For a
more detailed discussion, we recommend consulting Section 2 of \cite{MR2891702}.
\end{rem}

Note that elements in $H_{\tilde{K}}$ are $\mathbb{C}$-valued functions
defined on $S\times H$, especially, they have the following properties: 
\begin{cor}
\label{cor:B3}Any function $F\in H_{\tilde{K}}$ satisfies that 
\begin{equation}
F\left(s,a\right)=\left\langle a,V_{s}^{*}F\right\rangle _{H_{\tilde{K}}},\quad\left(s,a\right)\in S\times H.\label{eq:B10}
\end{equation}
In particular, $F$ is conjugate linear in the second variable, and
\[
F\left(s,0\right)=0,\quad\forall s\in S.
\]
\end{cor}

\begin{proof}
This follows from the reproducing property (see \eqref{eq:A3}) of
$H_{\tilde{K}}$: 
\[
F\left(s,a\right)=\left\langle \tilde{K}\left(\cdot,\left(s,a\right)\right),F\right\rangle _{H_{\tilde{K}}}=\left\langle V_{s}a,F\right\rangle _{H_{\tilde{K}}}=\left\langle a,V_{s}^{*}F\right\rangle _{H}.
\]
\end{proof}
\begin{rem}
In light of the above discussion (see Remark \ref{rem:B3}), \prettyref{cor:B3}
is a vector-valued reproducing property.
\end{rem}

\begin{cor}
\label{cor:B4}For any ONB $\left(\varphi_{i}\right)_{i\in\mathbb{N}}$
in $H_{\tilde{K}}$, the following operator identity holds:
\[
K\left(s,t\right)=\sum_{i\in\mathbb{N}}\left|V_{s}^{*}\varphi_{i}\left\rangle \right\langle V_{t}^{*}\varphi_{i}\right|.
\]
Here $\left|a\left\rangle \right\langle b\right|$ is Dirac's notation
for rank-1 operators, i.e., $c\mapsto a\left\langle b,c\right\rangle $. 
\end{cor}

\begin{proof}
Using the identity 
\[
I_{H_{\tilde{K}}}=\sum_{i}\left|\varphi_{i}\left\rangle \right\langle \varphi_{i}\right|
\]
in $H_{\tilde{K}}$, it follows that 
\[
K\left(s,t\right)=V_{s}^{*}V_{t}=V_{s}^{*}I_{H_{\tilde{K}}}V_{t}=\sum_{i}\left|V_{s}^{*}\varphi_{i}\left\rangle \right\langle V_{t}^{*}\varphi_{i}\right|.
\]

Equivalently, 
\begin{eqnarray*}
\left\langle a,K\left(s,t\right)b\right\rangle _{H} & = & \tilde{K}\left(\left(s,a\right),\left(t,b\right)\right)\\
 & \underset{\left(\ref{eq:A4}\right)}{=} & \sum_{i\in\mathbb{N}}\varphi\left(s,a\right)\overline{\varphi\left(t,b\right)}\\
 & = & \sum_{i\in\mathbb{N}}\left\langle V_{s}a,\varphi_{i}\right\rangle \left\langle \varphi_{i},V_{t}b\right\rangle \\
 & = & \sum_{i\in\mathbb{N}}\left\langle a,V_{s}^{*}\varphi_{i}\right\rangle \left\langle V_{t}^{*}\varphi_{i},b\right\rangle 
\end{eqnarray*}
valid for all $s,t\in S$, and all $a,b\in H$. 
\end{proof}

\subsection{\protect\label{subsec:b1}Examples}

The current setting of operator-valued p.d. kernels contains a wide
range of well-known constructions in operator theory (see e.g., \cite{MR2760647}
and \cite{MR1976867}). Passing from $K$ to a scalar-valued p.d.
kernel $\tilde{K}$ offers a function-based approach, in contrast
with traditional Hilbert completions with abstract spaces of equivalence
classes. This also allows for direct evaluation and manipulation of
functions, due to the RKHS structure of the dilation space $H_{\tilde{K}}$.

A classical example is when 
\[
K\left(s,t\right)=\rho\left(s^{-1}t\right),\quad s,t\in G
\]
where $\rho:G\rightarrow B\left(H\right)$ is a representation of
the group $G$. Note that, in this case, $K\left(s,s\right)=\rho\left(e\right)=I_{H}$,
the identity operator on $H$, for all $s\in G$. 
\begin{example}
As an application, we reformulate the well-known power dilation of
a single contraction $A\in B\left(H\right)$. Here, it is assumed
that $\left\Vert A\right\Vert \leq1$. Consider 
\[
A^{\left(n\right)}:=\begin{cases}
A^{n} & n\geq0\\
\left(A^{*}\right)^{n} & n<0
\end{cases}
\]
and set $K:\mathbb{Z}\times\mathbb{Z}\rightarrow B\left(H\right)$
by 
\[
K\left(m,n\right)=A^{\left(n-m\right)}.
\]
This $K$ is positive definite, i.e., $\left\langle h,Th\right\rangle \geq0$,
for all $h=\left(h_{1},\cdots,h_{n}\right)$ with $h_{i}\in H$ (see
\eqref{eq:A5}); or, equivalently, the matrix below is p.d.: 
\begin{equation}
T=\begin{bmatrix}1 & A & A^{2} & \cdots\\
A^{*} & 1 & A & \ddots\\
A^{*2} & A^{*} & 1 & \ddots\\
\vdots & \ddots & \ddots & \ddots
\end{bmatrix}_{n\times n}\label{eq:B12}
\end{equation}

Then, let $X=\mathbb{Z}\times H$ and define $\tilde{K}:X\times X\rightarrow\mathbb{C}$
by 
\[
\tilde{K}\left(\left(m,h\right),\left(n,k\right)\right)=\left\langle h,K\left(m,n\right)k\right\rangle _{H}=\left\langle h,A^{\left(n-m\right)}k\right\rangle _{H}.
\]
Denote by $H_{\tilde{K}}$ the RKHS of $\tilde{K}$. Set $V:H\rightarrow H_{\tilde{K}}$
by 
\[
Vh=\tilde{K}\left(\cdot,\left(0,h\right)\right).
\]
This is an isometric embedding. Set $U:H_{\tilde{K}}\rightarrow H_{\tilde{K}}$
by 
\[
U\tilde{K}\left(\cdot,\left(m,h\right)\right)=\tilde{K}\left(\cdot,\left(m+1,h\right)\right).
\]
Note, $U$ is unitary. It follows from \prettyref{thm:B1} that
\[
V^{*}U^{n}Vh=V^{*}U^{n}\tilde{K}\left(\cdot,\left(0,h\right)\right)=V^{*}\tilde{K}\left(\cdot,\left(n,h\right)\right)=A^{n}h,
\]
that is, 
\[
A^{n}=V^{*}U^{n}V.
\]
\end{example}

\begin{rem}
An elementary approach to verify that the operator matrix \eqref{eq:B12}
is positive definite is to rewrite $T$ as \renewcommand{\arraystretch}{1.8}
\begin{align*}
T & =\begin{bmatrix}1 & A & A^{2} & \cdots\\
A^{*} & 0 & 0 & \ddots\\
A^{*2} & 0 & 0 & \ddots\\
\vdots & \ddots & \ddots & \ddots
\end{bmatrix}+\begin{bmatrix}0 & 0 & 0 & \cdots\\
0 & 1 & A & \cdots\\
0 & A^{*} & 1 & \cdots\\
\vdots & \vdots & \vdots & \ddots
\end{bmatrix}\\
 & =\begin{bmatrix}1\\
A^{*}\\
\vdots\\
A^{*n}
\end{bmatrix}\begin{bmatrix}1 & A & \cdots & A^{n}\end{bmatrix}-\begin{bmatrix}0\\
A^{*}\\
\vdots\\
A^{*n}
\end{bmatrix}\begin{bmatrix}0 & A & \cdots & A^{n}\end{bmatrix}+\begin{bmatrix}0 & 0 & 0 & \cdots\\
0 & 1 & A & \cdots\\
0 & A^{*} & 1 & \cdots\\
\vdots & \vdots & \vdots & \ddots
\end{bmatrix}
\end{align*}
Therefore, 
\begin{eqnarray*}
\left\langle h,Th\right\rangle  & = & \left\Vert h_{1}+Ah_{2}+\cdots+A^{n}h_{n}\right\Vert ^{2}-\left\Vert Ah_{2}+\cdots+A^{n}h_{n}\right\Vert ^{2}\\
 &  & +\left\Vert h_{2}+Ah_{3}+\cdots+A^{n-1}h_{n}\right\Vert ^{2}-\left\Vert Ah_{3}+\cdots+A^{n-1}h_{n}\right\Vert ^{2}\\
 &  & +\left\Vert h_{3}+Ah_{4}+\cdots+A^{n-2}h_{n}\right\Vert ^{2}-\left\Vert Ah_{4}+\cdots+A^{n-2}h_{n}\right\Vert ^{2}\\
 &  & \vdots\\
 &  & +\left\Vert h_{n-1}+Ah_{n}\right\Vert ^{2}-\left\Vert Ah_{n}\right\Vert ^{2}+\left\Vert h_{n}\right\Vert ^{2}\\
 & \geq & \left\Vert h_{1}+Ah_{2}+\cdots+A^{n}h_{n}\right\Vert ^{2}\\
 &  & +\left(1-\left\Vert A\right\Vert ^{2}\right)\left[\left\Vert h_{2}+Ah_{3}+\cdots+A^{n-1}h_{n}\right\Vert ^{2}+\cdots+\left\Vert h_{n}\right\Vert ^{2}\right]\geq0.
\end{eqnarray*}
\end{rem}

Another notable example that falls in the general setting of \prettyref{thm:B1}
is the dilation of positive operator valued measures (POVMs), namely,
any POVM $Q:\mathscr{B}_{S}\rightarrow B\left(H\right)$ is the compression
of a projection valued measure (PVM) $P:\mathscr{B}_{S}\rightarrow B\left(L\right)$
for some Hilbert space $L\supset H$, where $\mathscr{B}_{S}$ denotes
the Borel $\sigma$-algebra of $S$. See, e.g., \cite{MR4705975,MR4509531,MR4280112,MR4256782}.
\begin{example}
Given $Q:\mathscr{B}_{S}\rightarrow B\left(H\right)$ as above, one
may define $K:\mathscr{B}_{S}\times\mathscr{B}_{S}\rightarrow B\left(H\right)$
by 
\[
K\left(A,B\right)=Q\left(A\cap B\right),\quad\forall A,B\in\mathscr{B}_{S},
\]
which is positive definite. Then set 
\begin{gather*}
\tilde{K}:\left(\mathscr{B}_{S},H\right)\times\left(\mathscr{B}_{S},H\right)\rightarrow\mathbb{C},\\
\tilde{K}\left(\left(A,a\right),\left(B,b\right)\right)=\left\langle a,Q\left(A\cap B\right)b\right\rangle _{H}
\end{gather*}
for all $A,B\in\mathscr{B}_{S}$, and all $a,b\in H$. 

Let $V:H\rightarrow H_{\tilde{K}}$ be the isometric embedding,
\[
Va=\tilde{K}\left(\cdot,\left(S,a\right)\right)
\]
where $H_{\tilde{K}}$ denotes the RKHS of $\tilde{K}$. It follows
that 
\[
V^{*}\tilde{K}\left(\cdot,\left(A,a\right)\right)=Q\left(A\right)a.
\]

Moreover, let $P:\mathscr{B}\left(S\right)\rightarrow H_{\tilde{K}}$
be specified by 
\[
P\left(A\right)\tilde{K}\left(\cdot,\left(B,a\right)\right)=\tilde{K}\left(\cdot,\left(A\cap B,a\right)\right).
\]
Then $P\left(A\right)^{*}=P\left(A\right)=P\left(A\right)^{2}$, i.e.,
a selfadjoint projection in $H_{\tilde{K}}$, for all $A\in\mathscr{B}_{S}$.
Finally, 
\begin{align*}
V^{*}P\left(A\right)Va & =V^{*}P\left(A\right)\tilde{K}\left(\cdot,\left(S,a\right)\right)\\
 & =V^{*}\tilde{K}\left(\cdot,\left(A\cap S,a\right)\right)=V^{*}\tilde{K}\left(\cdot,\left(A,a\right)\right)=Q\left(A\right)a.
\end{align*}
That is, $Q=V^{*}PV$, where $P$ is a PVM acting in $H_{\tilde{K}}$. 
\end{example}

For the general theory of dilations of irreversible evolutions in
algebraic quantum theory we refer to the two pioneering books \cite{MR489494,MR582649},
covering both the theory and applications of completely positive maps,
including semigroups.

\section{\protect\label{sec:3}Schwartz' constructions}

Motivated by applications, in the present section we extend the framework
for operator valued positive definite kernels to spaces more general
than Hilbert spaces.

We explore an approach to operator theory that extends beyond the
usual settings, moving into the realm of topological vector spaces
$L$ and their duals $L^{*}$. This idea was introduced to the first
named author by Professor D. Alpay, who may have been influenced by
L. Schwartz and others. 

Here, we look at positive definite (p.d.) kernels $K$ defined on
$S\times S$ with values in continuous linear operators from $L$
to $L^{*}$. This setup includes, but is not limited to, our previously
discussed $B(H)$-valued kernels. By extending to these more general
vector spaces, such as Banach spaces, we aim to broaden the application
of these kernels. We will show how our earlier models using $K$ can
be adapted to this extended context, particularly by transforming
these operator-valued kernels into scalar-valued kernels $\tilde{K}$
on the product space $X=S\times L$. This transition illustrates the
versatility and potential of our approach in various mathematical
and practical scenarios.

Let $S$ be a set, $L$ a topological vector space, and $L^{*}$ its
\emph{conjugate} dual space. We consider a positive definite (p.d.)
kernel $K:S\times S\rightarrow\mathcal{L}\left(L,L^{*}\right)$, where
$\mathcal{L}\left(L,L^{*}\right)$ denotes the space of continuous
linear operators from $L$ to $L^{*}$. 

Specifically, the p.d. assumption on $K$ is that 
\[
\sum_{i,j=1}^{n}\left\langle a_{i},K\left(s_{i},s_{j}\right)a_{j}\right\rangle \geq0
\]
for all $\left(s_{i}\right)_{i=1}^{n}$ in $S$, $\left(a_{i}\right)_{i=1}^{n}$
in $L$, and all $n\in\mathbb{N}$. Here, $\left\langle \cdot,\cdot\right\rangle :L\times L^{*}\rightarrow\mathbb{C}$
is the canonical dual pairing, i.e., $\left\langle \ell,\ell'\right\rangle :=\ell'\left(\ell\right)$,
for all $\ell\in L$, and $\ell'\in L^{*}$. 

Define $X=S\times L$, and 
\[
\tilde{K}\left(\left(s,a\right),\left(t,b\right)\right):=\left\langle a,K\left(s,t\right)b\right\rangle .
\]
Then $\tilde{K}$ is a scalar-valued p.d. kernel on $X\times X$;
indeed, we have 
\begin{align*}
\sum_{i,j=1}^{n}\overline{c_{i}}c_{j}\tilde{K}\left(\left(s_{i},a_{i}\right),\left(s_{j},a_{j}\right)\right) & =\sum_{i,j=1}^{n}\overline{c_{i}}c_{j}\left\langle a_{i},K\left(s_{i},s_{j}\right)a_{j}\right\rangle \\
 & =\sum_{i,j=1}^{n}\left\langle c_{i}a_{i},K\left(s_{i},s_{j}\right)c_{j}a_{j}\right\rangle \geq0.
\end{align*}
Let $H_{\tilde{K}}$ be the corresponding RKHS. 

Then we get the following variant of \prettyref{thm:B1}:
\begin{thm}
\label{thm:C1}The family of operators $\left\{ V_{s}\right\} _{s\in S}$,
where $V_{s}:L\rightarrow H_{\tilde{K}}$, defined as 
\[
V_{s}a=\tilde{K}\left(\cdot,\left(s,a\right)\right):X=S\times L\rightarrow\mathbb{C},\quad a\in L,
\]
satisfies the following:
\begin{enumerate}
\item For all $s\in S$, and all $a\in L$, 
\[
\left\Vert V_{s}a\right\Vert _{H_{\tilde{K}}}^{2}=\left\langle a,K\left(s,s\right)a\right\rangle .
\]
\item The adjoint $V_{s}^{*}:H_{\tilde{K}}\rightarrow L^{*}$ is determined
by 
\[
V_{s}^{*}\tilde{K}\left(\cdot,\left(t,b\right)\right)=K\left(s,t\right)b.
\]
\item $V_{s}V_{s}^{*}:H_{\tilde{K}}\rightarrow H_{\tilde{K}}$ is given
by 
\[
V_{s}V_{s}^{*}\tilde{K}\left(\cdot,\left(t,b\right)\right)=\tilde{K}\left(\cdot,\left(s,K\left(s,t\right)b\right)\right).
\]
\item For all $s_{1},s_{2},\cdots,s_{n}$ in $S$, 
\begin{multline*}
\left(V_{s_{1}}V_{s_{1}}^{*}\right)\cdots\left(V_{s_{n}}V_{s_{n}}^{*}\right)\tilde{K}\left(\cdot,\left(t,b\right)\right)\\
=\tilde{K}\left(\cdot,\left(s_{1},K\left(s_{1},s_{2}\right)\cdots K\left(s_{n-1},s_{n}\right)K\left(s_{n},t\right)b\right)\right).
\end{multline*}
\end{enumerate}
\begin{enumerate}[resume]
\item For all $s,s'\in S$, and $b\in L$, 
\[
V_{s'}V_{s}^{*}\tilde{K}\left(\cdot,\left(t,b\right)\right)=\tilde{K}\left(\cdot,\left(s',K\left(s,t\right)b\right)\right).
\]
\item $V_{s}^{*}V_{t}b=K\left(s,t\right)b$, and 
\[
\left(V_{s_{1}}^{*}V_{t_{1}}\right)\cdots\left(V_{s_{n}}^{*}V_{t_{n}}\right)b=K\left(s_{1},t_{1}\right)\cdots K\left(s_{n-1},t_{n-1}\right)K\left(s_{n},t_{n}\right)b.
\]
\end{enumerate}
\end{thm}

\section{\protect\label{sec:4}$H$-valued Gaussian processes}

In probability and stochastics, we think of Gaussian processes as
follows: Specify a probability space, and a system of random variables,
indexed by a set, say $S$, so $\left\{ W_{s},s\in S\right\} $. We
say that this is a Gaussian process, if the system $\left\{ W_{s},s\in S\right\} $
is jointly Gaussian. The corresponding covariance kernel will then
be positive definite (p.d.) , and by a standard construction, one
can show that every positive definite kernel arises this way. Note
that from first principles, every p.d. kernel on $S\times S$ arises
this way from a choice of Gaussian process $\left\{ W_{s},s\in S\right\} $.
In particular, it follows that Gaussian processes are determined by
their first two moments; where the p.d. kernel in question is interpreted
as second moments. The mean of the process is \textquotedblleft first
moments.\textquotedblright{} Suppose $\left\{ W_{h},h\in H\right\} $
is indexed by a Hilbert space $H$. If the inner product of $H$ coincides
with the second moments, i.e., the covariance kernel, we talk about
a Gaussian Hilbert space. See e.g., \cite{MR4641110}.

Since Ito's early work, the stochastic analysis of Gaussian processes
has often relied on Hilbert space techniques, as seen in \cite{MR3402823,MR4302453,MR45307}
and related literature. Here we want to call attention to the following
distinction: (i) scalar-valued Gaussian processes, potentially indexed
by a Hilbert space via an Ito-isometry, and (ii) Hilbert space-valued
Gaussian processes, see e.g., \cite{MR4641110,MR4414825,MR4101087,MR4073554,MR3940383}. 

We focus on the latter (ii), as it offers a more versatile approach
to model complex covariance structures in large datasets. This construction,
detailed below, is based on the results from \prettyref{sec:2}, utilizing
both $B(H)$-valued positive definite kernels, and their induced scalar
valued dilations, see \prettyref{thm:B1} and Corollary \ref{cor:B2}.

A key feature of Gaussian processes is that they are determined by
their first two moments, also when interpretated in a more general
metric framework. See e.g., \cite{MR4728479,MR4260603,MR3447225}.
\begin{defn}
Let $S$ be a set, and $H$ a Hilbert space. We say $\left\{ W_{s}\right\} _{s\in S}$
is an $H$-valued Gaussian process if, for all $a,b\in H$, and all
$s,t\in S$, 
\[
\mathbb{E}\left[\left\langle a,W_{s}\right\rangle _{H}\left\langle W_{t},b\right\rangle _{H}\right]=\left\langle a,K\left(s,t\right)b\right\rangle _{H}
\]
where $K$ is a $B\left(H\right)$-valued p.d. kernel. In particular,
\[
\left\langle W_{s},a\right\rangle _{H}\sim N\left(0,\left\langle a,K\left(s,s\right)a\right\rangle _{H}\right),
\]
i.e., mean zero Gaussian, with variance $\left\langle a,K\left(s,s\right)a\right\rangle _{H}$.
\end{defn}

Fix $S\times S\xrightarrow{\;K\;}B\left(H\right)$, and define $\tilde{K}:\left(S\times H\right)\times\left(S\times H\right)\rightarrow\mathbb{C}$,
given by 
\begin{equation}
\tilde{K}\left(\left(s,a\right),\left(t,b\right)\right)=\left\langle a,K\left(s,t\right)b\right\rangle _{H},\quad\forall s,t\in S,\ \forall a,b\in H,\label{eq:d1}
\end{equation}
defined on $S\times H$. Let $H_{\tilde{K}}$ be the corresponding
RKHS (consisting of scalar-valued functions on $S\times H$).
\begin{lem}
For all $\varphi\in H_{\tilde{K}}$, $s\in S$, and $a,b\in H$, we
have 
\begin{equation}
\varphi\left(s,a+b\right)=\varphi\left(s,a\right)+\varphi\left(s,b\right).\label{eq:d2}
\end{equation}
(See also Corollary \ref{cor:B3}.)
\end{lem}

\begin{proof}
It suffices to verify \eqref{eq:d2} for the generating functions
$\varphi\left(t,b\right)=\tilde{K}\left(\cdot,\left(t,b\right)\right)$
in $H_{\tilde{K}}$, which are functions defined on $S\times H$.
By definition, 
\begin{eqnarray*}
\varphi\left(t,b_{1}+b_{2}\right)\left(s,a\right) & = & \left\langle \tilde{K}\left(s,a\right),\tilde{K}\left(t,b_{1}+b_{2}\right)\right\rangle _{H_{\tilde{K}}}\\
 & \underset{\text{by \ensuremath{\left(\ref{eq:d1}\right)}}}{=} & \left\langle a,K\left(s,t\right)\left(b_{1}+b_{2}\right)\right\rangle _{H}\\
 & = & \left\langle a,K\left(s,t\right)b_{1}\right\rangle _{H}+\left\langle a,K\left(s,t\right)b_{2}\right\rangle _{H}\\
 &  & \text{since \ensuremath{K\left(s,t\right)\in B\left(H\right)}}\\
 & = & \varphi\left(t,b_{1}\right)\left(s,a\right)+\varphi\left(t,b_{2}\right)\left(s,a\right).
\end{eqnarray*}
\end{proof}
\begin{cor}
For any ONB $\left\{ e_{i}\right\} $ in $H$, we have the following
identity for functions $\varphi\in H_{\tilde{K}}$:
\begin{equation}
\sum_{i}\varphi\left(s,e_{i}\right)\left\langle e_{i},a\right\rangle _{H}=\varphi\left(s,a\right),\quad\forall\left(s,a\right)\in S\times H.\label{eq:d3}
\end{equation}
\end{cor}

\begin{proof}
Since $K\left(s,t\right)\in B\left(H\right)$ for all $\left(s,t\right)\in S\times S$,
\eqref{eq:d3} follows from the ONB representation in $H$, i.e.,
\[
\sum_{i}e_{i}\left\langle e_{i},a\right\rangle _{H}=a,\quad\forall a\in H,
\]
and 
\[
\sum\left|\left\langle e_{i},a\right\rangle _{H}\right|^{2}=\left\Vert a\right\Vert _{H}^{2}.
\]
Similarly, 
\[
\sum_{i}\left\langle a,e_{i}\right\rangle _{H}\left\langle e_{i},b\right\rangle _{H}=\left\langle a,b\right\rangle _{H},\quad\forall a,b\in H.
\]
\end{proof}
Let $K$, $\tilde{K}$, $S$, $H$ be as above. Introduce the operators
$\left(V_{s}\right)_{s\in S}$ from \prettyref{thm:B1}. Recall that
\begin{equation}
V_{s}^{*}V_{t}=K\left(s,t\right),\quad\forall s,t\in S.\label{eq:d7}
\end{equation}
The argument for the identity \eqref{eq:d7} relies on the pair of
operators, where 
\[
V_{t}a=\tilde{K}_{\left(t,a\right)}\left(\cdot\right):=\tilde{K}\left(\cdot,\left(t,a\right)\right)
\]
and 
\[
V_{s}^{*}\left(\tilde{K}_{\left(t,a\right)}\right)=K\left(s,t\right)a.
\]
 
\[
\xymatrix{ & H_{\tilde{K}}\ar[dr]^{V_{s}^{*}}\\
H\ar[rr]_{V_{s}^{*}V_{t}}\ar[ur]^{V_{t}} &  & H
}
\]
We verify that $V_{s}^{*}V_{t}=K\left(s,t\right)$, or equivalently,
\[
\left\langle V_{s}a,V_{t}b\right\rangle _{H_{\tilde{K}}}=\left\langle a,K\left(s,t\right)t\right\rangle _{H}
\]
for all $s,t\in S$, and all $a,b\in H$.

In what follows, we assume $H_{\tilde{K}}$ is separable. 
\begin{defn}
Consider the $H$-valued Gaussian process $W_{s}:\Omega\rightarrow H$,
\begin{equation}
W_{t}=\sum_{i}\left(V_{t}^{*}\varphi_{i}\right)Z_{i},\label{eq:d8}
\end{equation}
where $\left(\varphi_{i}\right)$ is an ONB in $H_{\tilde{K}}$. Following
standard conventions, here $\left\{ Z_{i}\right\} $ refers to a choice
of an independent identically distributed (i.i.d.) system of standard
scalar Gaussians $N(0,1)$ random variables, and with an index matching
the choice of ONB.
\end{defn}

\begin{thm}
\label{thm:B4}We have 
\begin{equation}
\mathbb{E}\left[\left\langle a,W_{s}\right\rangle _{H}\left\langle W_{t},b\right\rangle _{H}\right]=\left\langle a,K\left(s,t\right)b\right\rangle _{H}.\label{eq:d9}
\end{equation}
\end{thm}

\begin{proof}
Note that $\mathbb{E}\left[Z_{i}Z_{j}\right]=\delta_{i,j}$. From
this, we get
\begin{align*}
\text{LHS}_{\left(\ref{eq:d9}\right)} & =\mathbb{E}\left[\left\langle a,W_{s}\right\rangle _{H}\left\langle W_{t},b\right\rangle _{H}\right]\\
 & =\sum_{i}\sum_{j}\left\langle a,V_{s}^{*}\varphi_{i}\right\rangle \left\langle V_{t}^{*}\varphi_{j},b\right\rangle \mathbb{E}\left[Z_{i}Z_{j}\right]\\
 & =\sum_{i}\left\langle V_{s}a,\varphi_{i}\right\rangle _{H_{\tilde{K}}}\left\langle \varphi_{i},V_{t}b\right\rangle _{H_{\tilde{K}}}\\
 & =\left\langle V_{s}a,V_{t}b\right\rangle _{H_{\tilde{K}}}=\left\langle a,V_{s}^{*}V_{t}b\right\rangle _{H}=\left\langle a,K\left(s,t\right)b\right\rangle _{H}.
\end{align*}
\end{proof}
\begin{rem}
The central limit theorem, applied to the i.i.d. standard normal random
variables $\left\{ Z_{i}\right\} $ on the right-hand side of equation
\eqref{eq:d8}, ensures that $\left\{ W_{t}\right\} _{t\in S}$ is
a well-defined Gaussian process with the specified properties. 

While different choices of orthonormal bases $\left(\varphi_{i}\right)$
and i.i.d. standard $N\left(0,1\right)$ Gaussian random variables
$\left(Z_{i}\right)$ will result in distinct Gaussian processes $\left\{ W_{t}\right\} _{t\in S}$,
they all share the same covariance structure given by equation \eqref{eq:d9}.
Furthermore, once $\left(\varphi_{i}\right)$ and $\left(Z_{i}\right)$
are fixed, the corresponding Gaussian process $\left\{ W_{t}\right\} _{t\in S}$
is unique, since Gaussian processes are uniquely determined by the
first two moments. 
\end{rem}

\begin{rem}
Note that \eqref{eq:d8} gives a realization of the Gaussian process
$\left\{ W_{t}\right\} _{t\in S}$, assuming that $H_{\tilde{K}}$
is separable. In fact, the existence of such a Gaussian process satisfying
\eqref{eq:d9} is guaranteed without the separability condition. As
a consequence of \prettyref{thm:B4}, we obtain a second characterization
of operator-valued p.d. kernels, in terms of direct integral decompositions:
\end{rem}

\begin{cor}
Let $K:S\times S\rightarrow B\left(H\right)$. Then the kernel $K$
is p.d. if and only if it decomposes as
\[
K\left(s,t\right)=\int_{\Omega}\left|W_{s}\left\rangle \right\langle W_{t}\right|d\mathbb{P},
\]
where $\left\{ W_{t}\right\} _{t\in S}$ is the $H$-valued Gaussian
process from above. 
\end{cor}

\section*{Declarations}

The authors did not receive support from any organization for the
submitted work. The authors have no relevant financial or non-financial
interests to disclose.

\bibliographystyle{plain}
\nocite{*}
\bibliography{ref}

\end{document}